\documentclass[12pt]{article}

\usepackage{amssymb, amsmath, amsthm}
\usepackage{graphicx}
\usepackage{hyperref}
\hypersetup{
    colorlinks=true,
    linkcolor=blue,
    filecolor=magenta,      
    urlcolor=cyan,
}
\usepackage{dsfont}
\usepackage{enumerate}
\usepackage{stmaryrd}
\usepackage{mathtools}
\usepackage{fancyhdr}

\newcommand\Tr{\operatorname{Tr}}
\newcommand\B{\mathcal{B}}
\newcommand\Hp{\B(\bord)}
\newcommand\Hm{\B'(\bord)}
\newcommand\Hcal{\mathcal{H}}
\newcommand\U{\mathcal{U}}
\newcommand\J{\mathcal{J}}
\newcommand\Ocal{\mathcal{O}}
\newcommand\T{\mathcal{T}}
\newcommand\R{\mathbb{R}}
\newcommand\N{\mathbb{N}}
\newcommand\bord{\partial\Omega}
\newcommand\nbord{\R^n\backslash\bord}
\newcommand\dx{{\rm d}x}
\newcommand\dy{{\rm d}y}
\newcommand\dt{{\rm d}t}
\newcommand\dmu{{\rm d}\mu}
\newcommand\dps{\displaystyle}

\newcommand{\ddn}[1]{\frac{\partial #1}{\partial\nu}}
\newcommand{\ddny}[2]{\frac{\partial^{#2} #1}{\partial\nu}}

\newtheorem{theorem}{Theorem}[section]

\newtheorem{corollary}[theorem]{Corollary}

\theoremstyle{definition}
\newtheorem{definition}[theorem]{Definition}

\theoremstyle{remark}
\newtheorem{remark}[theorem]{Remark}

\numberwithin{equation}{section}

\providecommand{\keywords}[1]{\textbf{Keywords:} #1}

\author{
Gabriel Claret \thanks{Laboratoire MICS, CentraleSup\'elec, Universit\'e Paris-Saclay, France (correspondence: gabriel.claret@centralesupelec.fr).} \and
Anna Rozanova-Pierrat\thanks{Laboratoire MICS, CentraleSup\'elec, Universit\'e Paris-Saclay, France
(correspondence: anna.rozanova-pierrat@centralesupelec.fr).}}

\title{Existence of optimal shapes for heat diffusions across irregular interfaces}

\date{November, 2023}
\begin{document}

\maketitle

\begin{abstract}
We consider a heat transmission problem across an irregular interface -- that is, non-Lipschitz or fractal -- between two media (a hot one and a cold one). The interface is modelled as the support of a $d$-upper regular measure. We introduce the proprieties of the interior and exterior trace operators for two-sided extension domains, which allow to prove the well-posedness (in the sense of Hadamard) of the problem on a large class of domains, which contains regular domains, but also domains with variable boundary dimension. Then, we prove the convergence in the sense of Mosco of the energy form connected to the heat content of one of the domains and the heat transfer for $(\varepsilon,\infty)$-domains. Finally, we prove the existence of an optimal shape maximizing the heat energy transfer in a class of $(\varepsilon,\infty)$-domains, allowing fractal boundaries, while that optimum can generally not be reached in the class of Lipschitz domains.
\end{abstract}

\section{Introduction}

Optimizing the efficiency of heat transfers is a problem which often arises in the design of industrial objects, especially in the context of new technologies (for the cooling of microprocessors for example). It is linked to the heat propagation shortly after the beginning of the cooling of an object. Intuitively, the greater the area of the exchange surface, the most efficient the heat transfer (see Figs.~\ref{Fig:RegularInterface} and~\ref{Fig:IrregularInterface}, where the volumes of the initially cold and hot media are preserved, while the lengths of the interfaces increase). The idea that the shape of the interface between media has a significant impact on the speed of the diffusive heat transfer is well known from a physical perspective; for instance, it results from notable works by de Gennes~\cite{de_gennes_physique_1982}.
The heat propagation can be described by the increase of heat content of the initially cold medium. The matter of the efficiency of the heat transfer is then related with the asymptotic expansion of the heat content for small times $t$.
De Gennes~\cite{de_gennes_physique_1982} stated that the asymptotic behaviour of the heat content in an $n$-dimensional domain is proportional to the measure of the interface multiplied by $t^{(n-d)/2}$, where $d$ is the Hausdorff dimension of the interface.
The heat propagation for the small times is also directly connected to the area of the interior Minkowski sausage ($i.e.$ the $\sqrt{t}$-neighbourhood of the interface on the hot side). That argument was investigated experimentally and numerically in the case of prefractal configurations (see~\cite{mandelbrot_fractal_1983}) in~\cite{rozanova-pierrat_faster_2012}.  It was  established analytically in~\cite{bardos_short-time_2016} in the framework of regular and $d$-set boundaries for the heat transmission model between two media, described by different diffusion coefficients and an interface with finite or infinite resistivity (see~\eqref{Eq:DiffProb}).

For that matter, it is natural to consider shapes of finite volumes (due to industrial constraints) but with large or rough surfaces, which is why it appears consistent to think of objects with boundaries described by fractals~\cite{mandelbrot_how_1967}. Indeed, fractal shapes are self-similar (or scale invariant), which means they present irregularities at all scales, making the perimeter of the object infinite: the Hausdorff dimension of the boundary is non-integer.
Therefore, the short-time heat transfer is more efficient across fractal boundaries, the Hausdorff dimension of which is superior to $n-1$, which is the dimension of regular (for instance, Lipschitz) boundaries.
Figures~\ref{Fig:RegularInterface} and~\ref{Fig:IrregularInterface} illustrate that phenomenon, representing the cooling of a hot domain into a cold one through generation $0$ to $3$ Minkowski pre-fractal interfaces. In Figure~\ref{Fig:RegularInterface} (generations $0$ and $1$), the heat only spreads to the part of the cold domain sticking to the boundary. Nonetheless, the enhanced cooling near the angles of the generation $1$ curve is already visible. In line with that observation, the heat transfer appears more efficient as the interface grows more irregular: in Figure~\ref{Fig:IrregularInterface} (generations $2$ and $3$), the heat spreads further into the cold domain within the same time frame as in the previous cases, progressively reaching the overall shape of the fractal boundary rather than sticking to the interface. Consequently, the hot domain cools down more efficiently (represented by the yellow areas in the hot domains), and even more so for the generation $3$ interface.

\begin{figure}[ht]
\centering
\includegraphics[width=10cm]{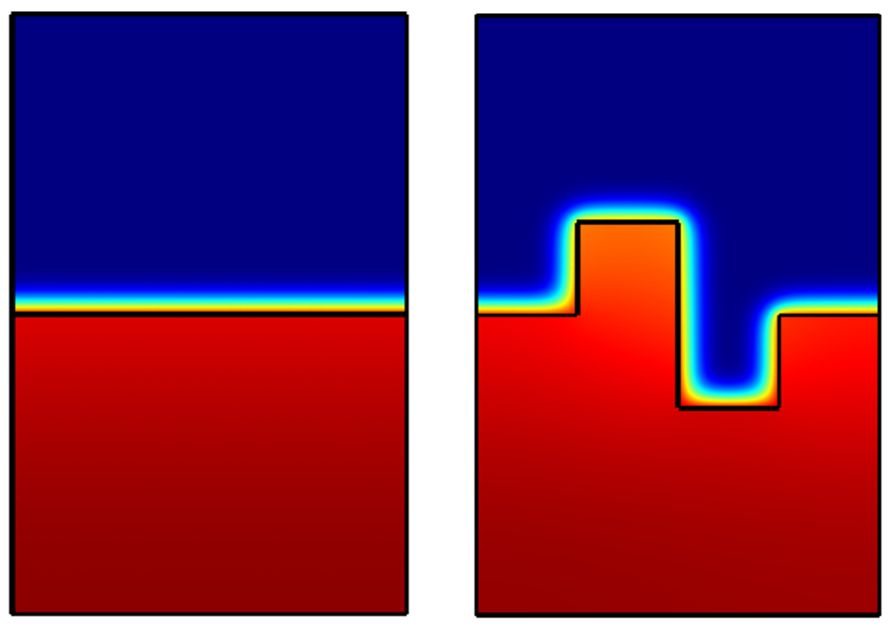}
\caption{Numerical simulation (with COMSOL Multiphysics) of the cooling of a hot medium (in red) by a cold medium (in blue) across a generation $0$ (on the left) and a generation $1$ (on the right) Minkowski pre-fractal curve. The propagation time is the same for both figures. The heat distributions satisfy homogeneous Neumann boundary conditions on the outer boundaries (a thermally insulated cavity) and the diffusion coefficients of two media are supposed to be distinct. The volume of the media is constant, due the symmetric property of Minkowski pre-fractal curve. }
\label{Fig:RegularInterface}
\end{figure}

\begin{figure}[ht]
\centering
\includegraphics[width=10cm]{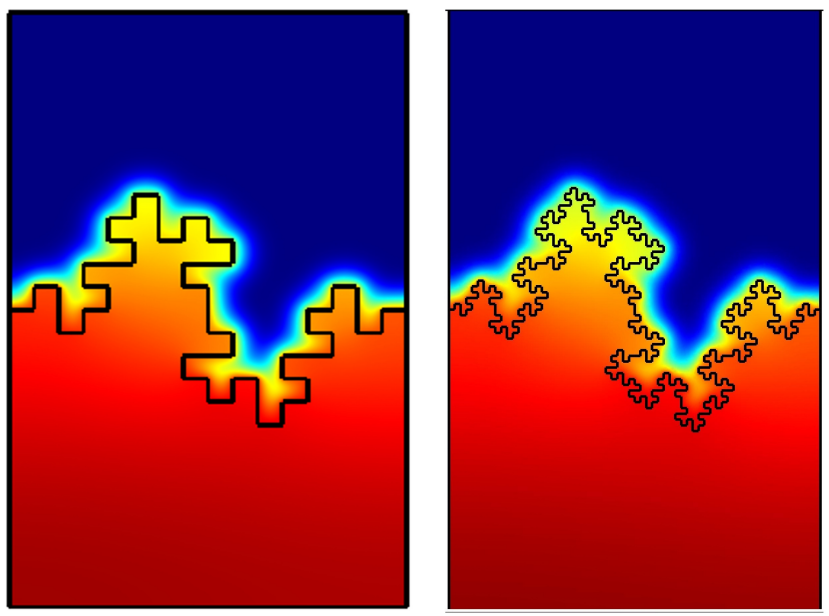}
\caption{Numerical simulation (COMSOL Multiphysics) of the cooling of a hot medium (in red) by a cold medium (in blue) across a generation 2 (on the left) and a generation 3 (on the right) Minkowski pre-fractal curve. The cooling time and boundary conditions are the same as for Figure~\ref{Fig:RegularInterface}.}
\label{Fig:IrregularInterface}
\end{figure}

There are many mathematical results on the asymptotic behaviour of the heat content in the case of one medium with a hot boundary (\textit{i.e.}, Dirichlet boundary condition): for a $C^3$ boundary~\cite{VAN_DEN_BERG-1994,VANDENBERG-1994}, for a polygonal boundary~\cite{VAN_DEN_BERG-1990} and for a triadic Von Koch snowflake~\cite{FLECKINGER-1995}. The special case of a countable disjoint union of scaled copies of a given open set in $\R^n$ is studied in~\cite{LEVITIN-1996}.
The case of a continuous metric (without discontinuity of the diffusion coefficients across the interface) on smooth compact
$n$-dimensional Riemannian manifolds with a smooth boundary was considered in~\cite{GILKEY-2003}.  The case of continuous
transmission boundary conditions for the expansion of the heat kernel
on the diagonal was studied in~\cite{PIROZHENKO-2005}. A survey of results on the asymptotic
expansion of the heat kernel for different boundary conditions can be found in~\cite{VASSILEVICH-2003}.

Although it is not the approach followed in this work, problems regarding heat diffusion have also been widely studied from a probabilistic perspective, relying on the connection between the analytical concept of Dirichlet forms and the probabilistic notion of Feller semi-group~\cite{chen_symmetric_2012}. 
More specifically, it is well known solutions to the heat equation can be expressed in terms of Brownian diffusions, see for instance~\cite[Paragraph 6]{hunt_theorems_1956} and~\cite{doob_probability_1955, rosenblatt_class_1951}. In those papers, the solution to the Dirichlet heat problem is expressed in terms of the parabolic measure. 
Given a point in the domain, the measure of a time interval and a Borel subset of the boundary can be expressed as the probability for a Brownian motion starting from that point to reach the boundary subset within the time interval (see also~\cite{wu_parabolic_1979}). 
With that approach, domain regularity is of little relevance: boundary points are split into regular points (\textit{i.e.}, starting points for Brownian motions hitting the boundary arbitrarily soon almost surely) and irregular points. The latter form a null set for the parabolic measure~\cite{doob_semimartingales_1954} and are connected to the thinness of the domain. 
Though the domain under study needs no specific regularity, the boundary conditions are seen as continuous functions. 
The analytical approach however demands some boundary regularity but allows less smooth conditions. For that matter, rather than strictly equivalent, the analytical and probabilistic approaches to the heat problem can be regarded as complementary.

Generally speaking, fractal geometries have been vastly studied for their potential industrial applications and their omnipresence in nature, see~\cite{mandelbrot_fractal_1983}, as well as~\cite{sapoval_fractal_1985} for instance. From a mathematical perspective, various types of heat models have been studied on irregular and fractal domains. In~\cite{capitanelli_fractional_2021}, the heat equation and other diffusion models are studied on random Koch snowflakes for Dirichlet, Neumann and Robin boundary conditions. 
Asymptotic results for the diffusion are obtained by approximation of the fractal shape by the associated pre-fractal sequence. In~\cite{lancia_venttsel_2014}, the existence and uniqueness of solutions to diffusion problems on Koch domains with Wentzell boundary conditions is proved. One can also think of~\cite{cefalo_fractal_2023} where the heat exchange interface consists of a so-called Koch-mixture. See~\cite{capitanelli_asymptotics_2010, lancia_second_2004, mosco_variational_2003} for other problems set on domains with Koch fractal boundaries. 
Those boundaries are examples of $d$-sets~\cite{jonsson_function_1984} for $d=\ln 4/\ln 3$, which have a fixed Hausdorff dimension (which is $d$). So far, from an analytical perspective, problems regarding heat diffusions on irregular shapes have been set on domains with at least $d$-set boundaries, see for instance~\cite{bardos_short-time_2016, creo_dynamic_2023}.  
Other boundary value problems have been studied on domains with $d$-set boundaries (see~\cite{arfi_dirichlet--neumann_2019} on the Poincaré-Steklov -- or Dirichlet-to-Neumann -- operator for Laplacian transport with Robin boundary conditions), but also boundaries with varying Hausdorff dimension.
Here, as in~\cite{hinz_existence_2021,dekkers_mixed_2022} (see also~\cite{lancia_fractals_2021} for a more detailed discussion), we use boundaries described as the support of a finite upper-regular Borel measure. 
One can think of~\cite{hinz_boundary_2023}, where Poisson-type problems are studied on domains with boundaries of bounded Hausdorff dimension, with Dirichlet, Neumann and Robin boundary conditions.
Recently, problems with Dirichlet and Neumann boundary conditions were even studied on extension domains, without a specified boundary measure~\cite{claret_functional_2023}. 
However, when the variational formulation of the problem considered contains a boundary integral (as for the heat transmission problem considered here), the choice of a boundary measure becomes crucial.
%\textcolor{red}{From a mathematical perspective, various types of boundary value problems have been studied on irregular and fractal domains. We refer to~\cite{capitanelli_fractional_2021, lancia_venttsel_2014} for the study of diffusions problems with Robin and Wenttsel type boundary conditions on Von Koch snowflakes for example; see also~\cite{capitanelli_asymptotics_2010, lancia_second_2004, mosco_variational_2003}.}

The heat transfer problem considered here can be modelled as follows. Consider a hot domain $\Omega^+$ with boundary $\bord$ and cold complement domain $\Omega^-$ (so that $\Omega^+$ and $\Omega^-$ are separated by their boundary $\bord$). The domain $\Omega^+$ is characterized by a thermal diffusivity $D^+>0$ and its complement by $D^->0$, where those coefficients are assumed to be distinct. 
The thermal resistance of the boundary is modelled by a continuous function $\Lambda:\bord\to[0,+\infty]$. 
If the heat distribution $u$ is initially unitary on $\Omega^+$ and null on $\Omega^-$ -- modelling the cooling of $\Omega^+$ by $\Omega^-$ -- then it is solution to the following problem:
\begin{equation}\label{Eq:DiffProb}
\begin{cases}
\dps\partial_t u^\pm- D^\pm \Delta u^\pm =0 &\mbox{on }]0,+\infty[\times\Omega^\pm,\\[0.5ex]
\dps D^-\ddn{u^-}=\Lambda(u^--u^+) &\mbox{on }]0,+\infty[\times\bord,\\[1.5ex]
\dps D^+ \ddn{u^+}=D^- \ddn{u^-} &\mbox{on }]0,+\infty[\times\bord,\\[1ex]
\dps u^+(0,\cdot)= 1 &\mbox{on }\Omega^+,\\[0.5ex]
u^-(0,\cdot)=0 &\mbox{on }\Omega^-.
\end{cases}
\end{equation}

To consider boundary value problems such as Problem~\eqref{Eq:DiffProb} on irregular shapes (for instance, fractals), the first concern regards the definition of the boundary conditions. On the one hand, the solution we seek is defined on $\Omega^+\cup\Omega^-$, which means it cannot simply be restricted to the boundary.
To generalize that idea, we use the trace operators. In the case of Lipschitz boundaries, that notion was defined and studied by Marschall in 1987~\cite{marschall_trace_1987} and the trace space (\textit{i.e.}, the range of the trace operators) corresponds to the Sobolev space $H^{1/2}$ on the boundary.
Shortly after, it was generalized to $d$-set boundaries (see~\cite{jonsson_function_1984, jonsson_dual_1995}) in~\cite{wallin_trace_1991}, in which case the trace space can be identified with the Besov space $B^{2,2}_\beta$, where $\beta=1-(n-d)/2$  and with $n$, the dimension of the ambient space.
Later on, the trace operator was defined for $d$-upper regular boundaries~\cite{JONSSON-2009, biegert_traces_2009, hinz_non-lipschitz_2021}, in which case it is proved the trace space is a Hilbert space, although it can no longer be identified with either one of the previous spaces. All three approaches rely on the existence of a boundary measure (Lebesgue's measure, a $d$-measure and a $d$-upper regular measure respectively) to define that operator.
In 2009, Biegert~\cite{biegert_traces_2009} relied on the extension property of $H^1$-Sobolev extension domains to define a trace operator which, in particular, does not depend on the choice of a boundary measure.
On the other hand, fractals are nowhere differentiable, which prevents from defining the normal derivatives in the usual way. For that matter, it is necessary to consider a weak definition of normal derivatives as in~\cite{lancia_transmission_2002} (see also~\cite{claret_functional_2023} on extension domains) via Green's formula, as elements of the dual of the trace space. 
Altogether, the notions of trace operators and weak normal derivatives, as well as the functional framework for domains with irregular boundaries from~\cite{chandler-wilde_sobolev_2017} allow to consider boundary value problems such as problem~\eqref{Eq:DiffProb} from a variational perspective on a broad class of domains (see Definition~\ref{Def:T-SAdDom}), including domains with fractal boundaries or with changing Hausdorff boundary dimensions.

In this article, we are interested in the existence of an optimal shape for the heat transfer problem between two media separated by a resistive boundary. The shape is optimal in the sense that it minimizes the total energy of the hot medium as the heat propagates during a fixed time $T>0$. 
In other words, we are seeking an optimal shape for the interface in some admissible class (first, in a class of uniform Lipschitz surfaces with uniformly bounded lengths and then, in a class of surfaces with uniformly bounded dimensions) to cool the initially hot medium in the most efficient way on a time interval $[0,T]$ (the problem is described by~\eqref{Eq:DiffProb} and the energy to minimize, by~\eqref{functionaltime}).
The main interest is for $T$ sufficiently small, however the results obtained hold for any fixed $T>0$. The difficulty of dealing with questions regarding the existence of optimal shapes is highlighted in~\cite{BUCUR-2005,BUCUR-2016} for problems with Robin boundary conditions, or as soon as the variational formulation displays an integral with respect to the boundary measure (due to non-homogeneous Robin or Neumann conditions). The main obstacle arises from the non-continuity of boundary perimeters (which are only lower semi-continuous~\cite[Proposition~2.3.6]{henrot_shape_2018}). 
As the shape of a boundary varies, so does its measure. The natural notion of convergence for those measures is the weak convergence. Mathematically speaking, the complexity results from the fact that the weak limit of a sequence of Hausdorff measures of co-dimension $1$ (which are the `natural' surface measures for Lipschitz boundaries) is not necessarily a Hausdorff measure itself. To overcome that problem in the framework of Lipschitz boundaries, relaxation methods have been developed using free-boundary problems with discontinuities in~\cite[Proposition~2.3.6]{henrot_shape_2018}. However, those methods do not preserve compactness properties and therefore, do not ensure the existence of an optimal shape. Historically, the framework of compact classes for the existence of optimal shapes (adapted to boundary problems with homogeneous Neumann conditions) is set by the famous results of Chenais~\cite{chenais_existence_1975}. That class of uniformly Lipschitz domains~\cite{chenais_existence_1975} is compact for the three usual types of domain convergence: in the sense of characteristic functions, compact sets and in the sense of Hausdorff. However, as for Robin-type boundary value problems, that class is not sufficient for our heat transfer problem.

That problem was recently solved in~\cite{magoules_optimal_2021,hinz_non-lipschitz_2021,hinz_existence_2021,hinz_boundary_2023}. In~\cite{magoules_optimal_2021} is proved the optimal wall in a class of Lipschitz domains for the absorption of acoustic waves is endowed with a boundary measure which is not necessarily the Hausdorff measure, while the Hausdorff measure is the only choice for such shapes for it measures the perimeter in a way which is consistent with physical observations.
It is proved in~\cite{hinz_non-lipschitz_2021} that the class of Lipschitz domains is not adapted to maximize the stability of an architectural structure for the same reason and that an optimal shape exists in a class of uniform domains. In both papers, the models under study are time-independant and set on a bounded domain alone. Here, we use a similar class of admissible shapes and adapt the method to the case of the heat diffusion problem, which is a time-dependant transmission problem.
Regarding the heat equation, the existence of a Lipschitz minimizer to the heat energy in a class of insulated Lipschitz domains with bounded volume is proved in~\cite{bucur_shape_2022}. The existence of an optimal heating boundary to maximize the heat diffusion is proved in~\cite{henrot_shape_1992} in the case of plane domains.
Still in the plane, it is proved in~\cite{trelat_optimal_2018} there exists an optimal shape to approximate a given heat distribution in a class of domains with a given number of holes and setting a homogeneous Dirichlet condition on the boundary.

The main results of this paper are the following: on a large class of domains, referred to as `two-sided admissible domains' and which can have boundaries with varying Hausdorff dimension, we introduce interior and exterior trace operators. We characterize them in terms of isometries with the Sobolev space $H^1$. We also identify their kernels as the closures of the sets of real-analytic functions with compact supports, as in the regular case (Theorem~\ref{Trisom}). We define interior and exterior weak normal derivatives in the same spirit as what is done in~\cite{lancia_transmission_2002} on $d$-sets. In the framework of our two-sided admissible domains, we prove the well-posedness (in the sense of Hadamard) of a generalized version of problem~\eqref{Eq:DiffProb}, adding a source term, considering general initial conditions and without assuming the thermal diffusivity is constant inside or outside (Theorem~\ref{Th:WP}). 
We introduce the energy form corresponding to the total heat content of the domain $\Omega^+$ throughout the diffusion and the heat transfer from $\Omega^+$ to $\Omega^-$, and prove the convergence of that energy form in the sense of Mosco as we approximate $(\varepsilon,\infty)$-domains~\cite{jonsson_function_1984} (Theorem~\ref{Th:Mosco}). Finally, we prove the class of Lipschitz domains is not adapted to optimize the heat transfer (Theorem~\ref{Th:ShapeOptLip}): optimal shapes should be endowed with boundary measures which are not necessarily Hausdorff measures. %, which is not consistent with physical observations.
For that matter, we extend the class of admissible domains to a class of $(\varepsilon,\infty)$-domains (or uniform domains), in which we prove the existence of an optimal shape in Theorem~\ref{Th:ShapeOptUnif}.

This paper is organised as follows: in Section~\ref{Sec:Framework-CV} we introduce the class of two-sided extension domains, the trace operators and the weak normal derivatives (Subsection~\ref{Sec:Framework}). We prove the isometric properties of the trace. In Subsection~\ref{Sec:WP}, we state a generalized version of problem~\eqref{Eq:DiffProb} and prove its well-posedness. In Section~\ref{Sec:CV-Opt}, we prove the Mosco convergence of the heat energy form (Subsection~\ref{Sec:Mosco}). Then, in Subsection~\ref{Sec:Opt}, we consider the shape optimization problem: we start with seeking an optimal shape in a class of Lipschitz domains and then turn to a class of $(\varepsilon,\infty)$-domains.

Throughout this paper, $n\ge 2$ is the dimension of the ambient space. If $A$ and $B$ are topological spaces, $C(A,B)$ denotes the set of all continuous functions mapping $A$ to $B$. All domains considered coincide with the interior of their closure.

\section{Framework and well-posedness}\label{Sec:Framework-CV}

\subsection{Functional framework}\label{Sec:Framework}

We begin with defining the general class of domains on which our study is carried out.

\begin{definition}[Two-sided admissible domains]\label{Def:T-SAdDom}
Let $\Omega$ be a bounded domain of $\R^n$ with connected boundary, such that the Hausdorff dimension of its boundary $\bord$ is (locally) included in $[n- 1, n[$.  Let $\mu$ be a positive finite Borel
measure with $\operatorname{supp}\mu=\bord$. The couple $(\Omega,\mu)$ is called a two-sided admissible domain if
\begin{enumerate}
\item[(i)] $\Omega^+=\Omega$ and its complement domain $\Omega^-:=\R^n\backslash\overline{\Omega}$ are $H^1$-extension domains of $\R^n$~\cite{hajlasz_sobolev_2008}:
there exist bounded linear extension operators $\mathrm{E}^\pm : H^1(\Omega^\pm)\to H^1(\mathbb{R}^n)$ with norms depending only on $\Omega$ (and $n$);
\item[(ii)]  the boundary measure $\mu$ is $d$-upper regular
for a fixed $d\in [n-1,n[$: \\ there exists $c_d>0$ such that
\begin{equation}\label{Eqmu}
\forall x\in\bord,\; \forall r\in]0,1],\quad\mu(B_r(x))\le c_d\,r^d,
\end{equation}
\end{enumerate}
where $B_r(x)$ is the open ball of $\R^n$ of center $x$ and radius $r$.
\end{definition}
Typical examples of two-sided admissible domains are NTA domains~\cite{nystrom_smoothness_1994, nystrom_integrability_1996}, $(\varepsilon,\infty)$-domains in $\R^2$~\cite{jones_quasiconformal_1981} and Lipschitz domains~\cite{calderon_lebesgue_1961, stein_singular_1970}. Condition~\eqref{Eqmu} implies that the Hausdorff dimension of the boundary is bounded below by $d$. However, unlike for $d$-sets, that dimension can vary, which allows to consider domains with boundaries which are partly fractal and partly Lipschitz for example (see Figure~\ref{Fig:SemiVK}). For Lipschitz boundaries, usual choices are $\mu=\lambda^{(n-1)}$, the $(n-1)$-dimensional Lebesgue measure and $\mu=\Hcal^{(n-1)}$, the $(n-1)$-dimensional Hausdorff measure.
\begin{figure}[ht]
\centering
\includegraphics[scale=0.5]{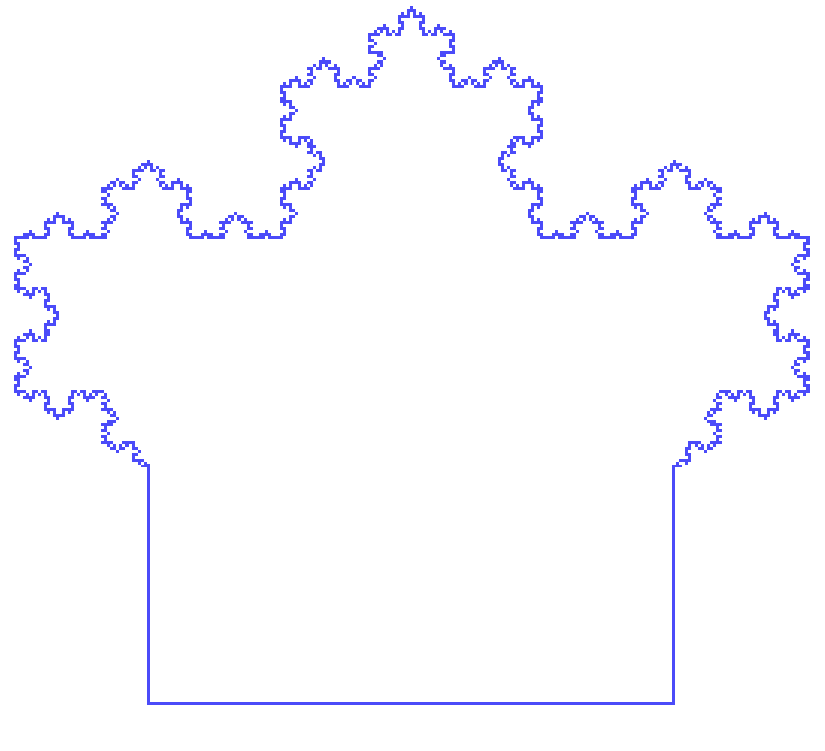}
\caption{An example of a two-sided  admissible domain of $\R^2$, lying inside the blue line. The upper part of the boundary is a Von Koch curve, which is a $\frac{\ln 4}{\ln 3}$-set (represented by a fifth generation pre-fractal {\small[source: Wikipedia]}). The lower part consists of a portion of a square, which is Lipschitz. The boundary is endowed with the 1-dimensional Lebesgue measure on the lower part and the $\frac{\ln 4}{\ln 3}$-dimensional Hausdorff measure on the upper part: condition~\eqref{Eqmu} is satisfied for $d=1$. }
\label{Fig:SemiVK}
\end{figure}

By~\cite[Theorem 5]{hajlasz_sobolev_2008}, condition (i) from Definition~\ref{Def:T-SAdDom} implies that $\Omega^\pm$ are $n$-sets ($d$-sets for $d=n$). Roughly speaking, that means the domain $\Omega$ should have neither inward nor outward cusps (see Figure~\ref{Fig:NonAdmissible}). 

\begin{figure}[ht]
\centering
\includegraphics[scale=0.9]{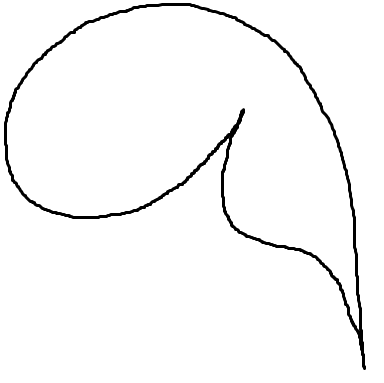}
\caption{An example of a domain $\Omega$ which is not a two-sided extension domain due to the presence of an inward and an outward cusp: neither $\Omega^+$ nor $\Omega^-$ is an $H^1$-extension domain.}
\label{Fig:NonAdmissible}
\end{figure}

In this section, whenever $(\Omega,\mu)$ is a two-sided admissible domain of $\R^n$, we use the following notations:
\begin{itemize}
\item[\raisebox{2pt}{\tiny $\bullet$}] $\Omega^+:=\Omega$ is thought of as the hot domain and its complement domain $\Omega^-:=\R^n\backslash\overline{\Omega}$ is thought of as cold;
\item[\raisebox{2pt}{\tiny $\bullet$}] if $u:(t,x)\in\, ]0,+\infty[\times\R^n\longmapsto u(t,x)\in\R$, then $u^\pm:=u|_{(]0,+\infty[\times\Omega^\pm)}$ are its restrictions to $\Omega^+$ and $\Omega^-$ respectively.
\end{itemize}

The $H^1$-extension property of $\Omega^\pm$ ensures the existence of bounded linear interior and exterior trace operators from $H^1(\Omega^+)$ and $H^1(\Omega^-)$ respectively (as in~\cite{claret_functional_2023}).
The interior trace operator is defined in the same way as in~\cite[Section 5]{biegert_traces_2009} (see also \cite{hinz_non-lipschitz_2021}). Note that its definition does not depend on the way elements of $H^1(\Omega^+)$ are extended to $H^1(\R^n)$~\cite[Theorem 6.1]{biegert_traces_2009}. Therefore, as in the $d$-regular case~\cite[Theorem 1]{wallin_trace_1991}, we use the following equivalent definition:
\begin{definition}[Trace operators]\label{DefTraceExt}
Let $(\Omega,\mu)$ be a two-sided admissible domain of $\R^n$.

\noindent \raisebox{2pt}{\tiny $\bullet$}\; The interior trace operator $\Tr^+:H^1(\Omega^+)\to \Tr^+(H^1(\Omega^+))\subset L^2(\bord,\mu)$ is defined by
\begin{equation}\label{E:pointwiseredef}
\Tr^+u(x)=\displaystyle\lim_{r\to 0^+}\frac{1}{\lambda^{(n)}(\Omega^+\cap B_r(x))}\int_{\Omega^+\cap B_r(x)}{u(y)\,\dy},\quad x\in \bord \mbox{ q.e.},
\end{equation}
where $\lambda^{(n)}$ denotes Lebesgue's measure on $\R^n$. % and $B_r(x)$ is the open ball of $\R^n$ of center $x$ and radius $r$.

\noindent \raisebox{2pt}{\tiny $\bullet$}\; The exterior trace operator $\Tr^-:H^1(\Omega^-)\to \Tr^-(H^1(\Omega^-))\subset L^2(\bord,\mu)$ is defined by
\begin{equation}\label{EqDefTrmin}
\Tr^-u(x)=\displaystyle\lim_{r\to 0^+}\frac{1}{\lambda^{(n)}(\Omega^-\cap B_r(x))}\int_{\Omega^-\cap B_r(x)}{u(y)\,\dy}, \quad x\in\bord\mbox{ q.e.}
\end{equation}
\end{definition}

In that definition, the term quasi-everywhere (q.e.) is understood in terms of the capacity $\mathrm{Cap}_2$ from~\cite[Section 2]{biegert_traces_2009}. Unlike the space $L^2(\bord,\mu)$, the range of the trace operators does not depend on $\mu$. As for~\cite[Theorem 5.1]{hinz_existence_2021}, it follows from~\cite[Corollaries 7.3 and 7.4]{biegert_traces_2009} and the fact that the measure on $\bord$ is finite (see~\cite[Theorems 7.2.2 and 7.3.2]{adams_function_1996}) that the limits~\eqref{E:pointwiseredef} and~\eqref{EqDefTrmin} exist $\mu$-a.e.

The following theorem summarizes the main properties of the trace operators and a characterization of their range (see~\cite{claret_functional_2023}).

\begin{theorem}[Trace theorem]\label{Trisom}
Let $(\Omega,\mu)$ be a two-sided admissible domain of $\R^n$. Then, the following assertions hold.
\begin{enumerate}
\item\label{Trisom0} The trace operators $\Tr^\pm:H^1(\Omega^\pm)\to L^2(\bord,\mu)$ are linear and compact. % while 
\item\label{Trisom1} The image of the trace operators
\begin{equation*}
\B(\bord):=\Tr^+(H^1(\Omega^+))=\Tr^-(H^1(\Omega^-))
\end{equation*}
is a Hilbert space endowed with either one of the equivalent norms:
\begin{equation}\label{normTr}
\left\|f\right\|_{\Tr^\pm}:=\min\{ \left\|v\right\|_{H^1(\Omega^\pm)}|\ f=\mathrm{Tr}^\pm\:v\}.
\end{equation}
The embedding $\B(\bord)\subset L^2(\bord,\mu)$ is dense and compact. In what follows, we denote $\|\cdot\|_{\Hp}:=\|\cdot\|_{\Tr^+}$. 
\item\label{Trisom2} $\operatorname{Ker}(\Tr^\pm)=H^1_0(\Omega^\pm):=\overline{C^\infty_0(\Omega^\pm)}^{\|\cdot\|_{H^1(\Omega^\pm)}}$ is the closure in $H^1$ of the set of indefinitely differentiable functions with compact support in $\Omega^\pm$.
\item\label{Trisom3} The trace operators $\Tr^\pm: H^1(\Omega^\pm) \to (\Hp,\|\cdot\|_{\Tr^\pm})$ are  partial isometries with operator norms equal to $1$.
\item\label{Trisom4} $\Tr^\pm: V_1(\Omega^\pm) \to \Hp$ define isometries where 
\begin{equation}\label{V1}
V_1(\Omega^\pm)=\big\{v\in H^1(\Omega^\pm)\;\big|\; (-\Delta+1)v=0 \mbox{ weakly}\big\}
\end{equation}
are the spaces of $1$-harmonic functions on $\Omega^\pm$ endowed with the standard $H^1$ norms. It holds:
\begin{equation}\label{EqV1andOthers}
V_1(\Omega^\pm)=\operatorname{Ker}(\Tr^\pm)^\perp =(H^1_0(\Omega^\pm))^\perp.
\end{equation}
\end{enumerate}
\end{theorem}

\begin{proof}
Point \eqref{Trisom1} follows partially from the definition, see also~\cite[Theorem~1]{hinz_existence_2021},\cite[Theorem~5.1]{hinz_non-lipschitz_2021},~\cite[Theorem~2]{dekkers_mixed_2022}, where the norm was defined as:
$$\left\|f\right\|_{\operatorname{Tr}^+}=\inf\{\left\|v\right\|_{H^1(\Omega^+)}|\ f=\mathrm{Tr}^+\:v\}.$$
Here, we specify that the infimum is reached. To prove that, it is sufficient to use the continuity of the trace operator from $H^1(\Omega)$ to its image and Stampacchia's theorem~\cite[Theorem V.6]{brezis_analyse_1987}, ensuring the existence of $u\in H^1_f(\Omega)$, where 
\begin{equation}\label{H1f}
H^1_f(\Omega^+):=\left\{v\in H^1(\Omega^+)\;\big|\;\Tr^+v=f \hbox{ on } \bord \right\},
\end{equation}
unique minimizer of $\|\cdot\|_{H^1(\Omega^+)}$ on $H^1_f(\Omega^+)$. It is equivalent to the weak well-posedness on $H^1(\Omega^+)$ of the non-homogeneous Dirichlet boundary value problem for $(-\Delta+1)$: 
\begin{equation}\label{DO}
\begin{cases}
-\Delta u+u=0&\mbox{on }\Omega^+,\\
\Tr^+u=f.
\end{cases}
\end{equation}
The weak solution to problem~\eqref{DO} is understood in the sense of its variational formulation
 on $H^1_f(\Omega^+)$:
\begin{equation}\label{DOvaralt}
\forall v\in H^1_f(\Omega^+),\quad \langle u,v\rangle_{H^1(\Omega^+)}=\|u\|^2_{H^1(\Omega^+)},
\end{equation}
which, in particular, implies:
\begin{equation}\label{EqEstH1}
\forall v\in H^1_f(\Omega^+),\quad\Big(\|v\|_{H^1(\Omega^+)}=\|f\|_{\Hp}\quad \iff\quad u=v\Big).
\end{equation}
To prove Point \eqref{Trisom2}, we use the following characterization from~\cite[Corollary 2.3.1 and Example 2.3.1]{fukushima_dirichlet_2010}:
\begin{equation*}
H^1_0(\nbord)=\big\{v\in H^1(\R^n)\;\big|\;\tilde{v}=0\mbox{ q.e. on }\bord\big\}.
\end{equation*}
$u\in H^1(\Omega^+)$ can be extended to $H^1_0(\nbord)$ if and only if $u\in H^1_0(\Omega^+)$, and it can be extended to the space on the right hand side of the equality if and only if $\Tr^+u=0$.
Points \eqref{Trisom3} and \eqref{Trisom4} are corollaries of the previous points, of 
Stampacchia's theorem in particular. Let us prove~\eqref{EqV1andOthers}.
For $u\in H^1(\Omega^+)$, by Green's formula (see Definition~\ref{def:ddn} below), the following equivalences hold:
\begin{align*}
u\in V_1(\Omega^+) &\iff -\Delta u+u=0\mbox{ weakly on }\Omega^+\\
&\iff \forall v\in \operatorname{Ker}(\Tr^+),\quad\langle u,v\rangle_{H^1(\Omega^+)}=0\\
&\iff u\in \operatorname{Ker}(\Tr^+)^\perp.
\end{align*}
Since $\Tr^+: \operatorname{Ker}(\Tr_i)^\perp \to \Hp$ is bijective, by~\eqref{EqEstH1}, it is an isometry. Therefore, since $\operatorname{Ker}(\Tr^+)^\perp\subset H^1(\Omega^+)$ is a closed Hilbert space endowed with $\|\cdot\|_{H^1}$ (as an orthogonal complement), the trace operator $\Tr^+: H^1(\Omega^+)\to \Hp$ is a partial isometry of operator norm equal to $1$. The results regarding $\Tr^-$ are proved in the same way.
\end{proof}

\begin{remark}
As mentioned before, the trace space $\B(\bord)$ does not depend on the boundary measure $\mu$.
If $\bord$ is Lipschitz, then $\B(\bord)=H^{1/2}(\bord)$ in the sense of equivalent norms.
If $\bord$ is a $d$-set for some $n-2<d<n$, then~\cite{jonsson_function_1984,wallin_trace_1991}  $\B(\bord)=B^{2,2}_\beta(\bord)$ with $\beta=1-\frac{n-d}{2}>0$.
\end{remark}

To understand problem~\eqref{Eq:DiffProb}, we also need to define normal derivatives on $\Omega$. The class of two-sided admissible domains contains fractal shapes, which are known for being nowhere differentiable. For that matter, we cannot use the usual definition, which relies on the existence of a normal vector. For that matter, we use the weak definition of normal derivatives as elements of the dual of the trace space, similarly to what is done in~\cite{lancia_transmission_2002}.

\begin{definition}[Weak normal derivatives]\label{def:ddn}
Consider a two-sided admissible domain $(\Omega,\mu)$ of $\R^n$.

\noindent \raisebox{2pt}{\tiny $\bullet$}\; Let $u\in H^1_\Delta(\Omega^+)$, where
\begin{equation}\label{Eq:H1D}
H^1_\Delta(\Omega^+):=\big\{v\in H^1(\Omega^+)\;\big|\;\Delta v\in L^2(\Omega^+)\big\}.
\end{equation}
The weak interior normal derivative of $u$ is the element $\varphi\in\Hm$ such that:
\begin{equation}
\forall v\in H^1(\Omega^+),\quad \left\langle\varphi, \Tr^+v\right\rangle_{\B'\!,\,\B}=\int_{\Omega^+}{(\Delta u)v\,\dx}+\int_{\Omega^+}{\nabla u\cdot\nabla v\,\dx},
\end{equation}
denoted by $\dps\ddny{u}+:=\varphi$.\\[0.5ex]

\noindent \raisebox{2pt}{\tiny $\bullet$}\; Let $u\in H^1_\Delta(\Omega^-)$ (\eqref{Eq:H1D} with $\Omega^-$ instead of $\Omega^+$). The weak exterior normal derivative of $u$ is the element $\psi\in\Hm$ such that:
\begin{equation}
\forall v\in H^1(\Omega^-),\quad \left\langle\psi, \Tr^-v\right\rangle_{\B'\!,\,\B}=-\int_{\Omega^-}{(\Delta u)v\,\dx}-\int_{\Omega^-}{\nabla u\cdot\nabla v\,\dx},
\end{equation}
denoted by $\dps\ddny{u}-:=\psi$.
\end{definition}

The sign convention in the definition of the weak exterior normal derivative corresponds, in the regular case, to choosing the outward normal vector to $\Omega^+$ to define both derivatives. The boundary conditions in the heat transfer equation from problem~\eqref{Eq:DiffProb} are boundary jumps, defined as follows in the case of two-sided admissible domains.

\begin{definition}[Boundary jump values]
Consider a two-sided admissible domain $(\Omega,\mu)$ of $\R^n$ and a thermal diffusivity $D\in \T(\R^n)$, where
\begin{multline}
\T(\R^n):=\Big\{D\in L^\infty(\R^n)\;\Big|\;\underset{\R^n}{\mathrm{ess}\inf} D>0 \quad \mbox{ and }\\
\exists V \mbox{ neighbourhood of }\bord,\; D|_{V\backslash\bord}\in C(V\backslash\bord,]0,+\infty[)\Big\}.
\end{multline}
Let $u\in V(\R^n)$, where
\begin{equation}
V(\R^n):=\left\{v\in L^2(\R^n)\;\left|\;v^+\in H^1(\Omega^+),\;v^-\in H^1(\Omega^-)\right.\right\},
\end{equation}
endowed with the norm:
\begin{equation}\label{Eq:NormV}
\|\cdot\|_{V(\R^n)}:=\sqrt{\|\cdot\|^2_{L^2(\R^n)}+\|D\nabla\cdot\|_{L^2(\Omega^+\cup\Omega^-)^n}^2}.
\end{equation}
The jump in trace of $u$ across $\bord$ is defined by:
\begin{equation*}
\llbracket\Tr u\rrbracket := \Tr^+ u-\Tr^-u.
\end{equation*}

In addition, if $u^+\in H^1_\Delta(\Omega^+)$ (Eq.~\eqref{Eq:H1D}) and $u^-\in H^1_\Delta(\Omega^-)$, then the jump in co-normal derivative of $u$ is denoted by:
\begin{equation*}
\left\llbracket D\ddn{u}\right\rrbracket := D^+\ddny{u}+-D^-\ddny{u}-, 
\end{equation*}
where, for $x\in\bord$, $\dps D^\pm(x)=\lim_{\begin{smallmatrix}y\to x\\y\in\Omega^\pm\end{smallmatrix}}D(y)$.
\end{definition}

Let $(\Omega,\mu)$ be a two-sided admissible domain of $\R^n$. We consider the following heat transmission problem:
\begin{equation}\label{Eq:Problem}
\begin{cases}
\dps\partial_t u-D \Delta u =f &\mbox{on }]0,+\infty[\times(\nbord),\\[0.5ex]
\dps D^-\ddny{u}{-}+\Lambda\llbracket\Tr u\rrbracket=0 &\mbox{on }]0,+\infty[\times\bord,\\[1.5ex]
\!\dps\left\llbracket D\ddn{u}\right\rrbracket=0 &\mbox{on }]0,+\infty[\times\bord,\\[1ex]
u(0,\cdot)=u_0 &\mbox{on }\nbord,
\end{cases}
\end{equation}
where $u_0\in L^2(\R^n)$, $f\in L^2(]0,+\infty[,L^2(\R^n))$, $D\in \T(\R^n)$ and $\Lambda\in C(\bord,[0,+\infty])$. Those quantities can be interpreted as follows:
\begin{itemize}
\item[\raisebox{2pt}{\tiny $\bullet$}] $u$ is the (time-dependent) heat distribution on $\R^n$;
\item[\raisebox{2pt}{\tiny $\bullet$}] $u_0\in L^2(\R^n)$ is the initial heat distribution. In the example of problem~\eqref{Eq:DiffProb}, $u_0=\mathds{1}_{\Omega^+}$: the domain $\Omega^+$ is initially hot, while its complement is cold, and the problem models the heat diffusion from $\Omega^+$ to $\Omega^-$;
\item[\raisebox{2pt}{\tiny $\bullet$}] $f\in L^2(]0,+\infty[,L^2(\R^n))$ is the (time-dependent) heat source on $\R^n$;
\item[\raisebox{2pt}{\tiny $\bullet$}] $D\in \T(\R^n)$ is the thermal diffusivity. In the example of problem~\eqref{Eq:DiffProb}, $D=D^+\mathds{1}_{\Omega^+}+D^-\mathds{1}_{\Omega^-}$ with $D^+\neq D^-$, which models the diffusion properties of two homogeneous media;
\item[\raisebox{2pt}{\tiny $\bullet$}] $\Lambda\in C(\bord,[0,+\infty])$ is the resistivity of the heat conduction by the boundary. If $\Lambda=0$, then the boundary $\bord$ acts as a thermal insulator. If $\Lambda=+\infty$, then the heat distribution $u$ is continuous across the boundary.
\end{itemize}
In the case of a Lipschitz domain, the measure $\mu$ is chosen as the Hausdorff measure $\Hcal^{(n-1)}$, which is the usual way to measure the perimeter of those domains. With that choice, the model is consistent with the physical phenomenon.

\subsection{Well-posedness of the model}\label{Sec:WP}

We begin with proving that our problem is well-posed. 

Let $u_0\!\in\! L^2(\R^n)$, $f\!\in\! L^2(]0,+\infty[,L^2(\R^n))$, $D\!\in\! \T(\R^n)$ and $\Lambda\!\in\! C(\bord,[0,+\infty])$. Problem~\eqref{Eq:Problem} is understood in term of its variational formulation, given by:
\begin{equation}\label{Eq:FormVar}
\begin{cases}
\forall t>0,\;\forall v\in V(\R^n),\quad \langle\partial_tu(t),v \rangle_{L^2(\R^n)} + a(u(t),v)=\langle f(t),v\rangle_{L^2(\R^n)},\\
u(0,\cdot)=u_0\in L^2(\R^n),
\end{cases}
\end{equation}
where
\begin{equation}\label{Eq:a}
a(u,v):=\int_{\nbord}{D\nabla u\cdot\nabla v\,\dx}+\int_{\{\Lambda<+\infty\}}{\Lambda\llbracket\Tr u\rrbracket\llbracket\Tr v\rrbracket\,\dmu}.
\end{equation}

\begin{theorem}[Well-posedness]\label{Th:WP}
Let $(\Omega,\mu)$ be a two-sided admissible domain of $\R^n$. Let $u_0\!\in\! L^2(\R^n)$, $f\!\in\! L^2(]0,+\infty[,L^2(\R^n))$, $D\!\in\! \T(\R^n)$ and $\Lambda\!\in\! C(\bord,[0,+\infty])$.
Then, there exists a unique solution $u$ which is in the space $H^1(\R_+,L^2(\R^n))\cap L^2(\R_+,V(\R^n))$ to the variational formulation (Eq.~\eqref{Eq:FormVar}).\\
For all $t\ge0$, that solution satisfies the following energy identity:
\begin{equation}\label{Eq:EnergyEstimate}
\frac{1}{2}\int_{\R^n}|u(t)|^2\,\dx+\int_0^t a(u(\tau),u(\tau))\,\mathrm{d}\tau=\frac{1}{2}\int_{\R^n}{|u_0(x)|^2\,\dx}+\int_0^t{f(\tau)\,\mathrm{d}\tau}.
\end{equation}
\end{theorem}

\begin{proof}
For all $T>0$, the existence and uniqueness of the solution $u$ in the space $H^1([0,T],L^2(\R^n))\cap L^2(]0,T[,V(\R^n))$ to formulation~\eqref{Eq:FormVar} on the time interval $[0,T]$ for $T>0$ given $u_0\in L^2(\R^n)$ is proved as in~\cite[Paragraph 7.1.2]{evans_partial_2010}. Here, the bilinear form $a:V(\R^n)\times V(\R^n)\to\R$ satisfies the following coercivity estimate:
\begin{equation*}
\forall v\in V(\R^n),\quad a(v,v)\ge \|v\|^2_{V(\R^n)} - \|v\|^2_{L^2(\R^n)}.
\end{equation*}
The uniqueness of the solution on $[0,T]$ for all $T>0$ (given $u_0\in L^2(\R^n)$) proves the solutions on each time interval are compatible: if $u_1$ is a solution on $[0,T_1]$ and $u_2$ is a solution on $[0,T_2]$, then $u_1=u_2$ on $[0,T_1\wedge T_2]$. That allows to define a unique global solution $u\in H^1(\R_+,L^2(\R^n))\cap L^2(\R_+,V(\R^n))$.
\end{proof}

\begin{remark}
Since $H^1(\R_+,L^2(\R^n))\subset C(\R_+,L^2(\R^n))$, $u$ is also an element of $C(\R_+,L^2(\R^n))$.
\end{remark}

The problem being well-posed, the heat distribution is unique given the parameters of the problem and it is possible to consider the properties of the energy form related to the heat content in $\Omega^+$ and the heat transfer from $\Omega^+$ to $\Omega^-$.

\section{Convergence and optimization}\label{Sec:CV-Opt}

In this section, we consider the energy form associated to the heat content inside a domain $\Omega$ and its convergence as we approximate $\Omega$. Throughout this section, the diffusion time $T>0$ is fixed, $U$ denotes a (large) bounded Lipschitz domain and $\Lambda\in C(U,[0,+\infty])$. In what follows, we assume $\Lambda<+\infty$ on $U$ to simplify the expressions. Retrieving the general case can be done by replacing `$\bord$' with `$\bord\cap\{\Lambda<+\infty\}$' in the definition of the energy form (Eq.~\eqref{functionaltime} below). All the results stated hereafter remain valid once that modification has been made. We assume the thermal diffusivity $D$ is a constant $D^+>0$ on all the interior domains, and another constant $D^->0$ on the exterior domains, as in the case of problem~\eqref{Eq:DiffProb}. Without loss of generality, we set $D^+=1$.

\subsection{Mosco convergence of the energy forms}\label{Sec:Mosco}

If $(\Omega,\mu)$ is a two-sided admissible domain of $\R^n$ with $\Omega\subset U$, we define the energy functional $J_T(\Omega,\mu)$ on $L^2(]0,T[,L^2(U))$ by:

\begin{multline}\label{functionaltime}
J_T(\Omega, \mu)(u):=\int_0^T\!\!\int _{\Omega}| \nabla u|^2\,\dx\,\dt+\int_0^T\!\!\int_{\Omega}| u|^2\,\dx\,\dt\\
+\int_0^T\!\!\int_{\bord}\Lambda\left|\llbracket\Tr u\rrbracket\right|^2\,\dmu\,\dt.
\end{multline}

We prove the convergence of the energy form $J_T$ defined by~\eqref{functionaltime} in the sense of Mosco~\cite[Definition 2.1.1]{mosco_composite_1994} along an approximating sequence of domains, defined as follows.

\begin{definition}[Mosco convergence]\label{Def:MoscoCV}
Let $H$ be a Hilbert space. Let $(F_m)_{m\in\N}$ be a sequence of quadratic forms on $H$. $(F_m)_{m\in\N}$ is said to converge in the sense of Mosco (or $M$-converge) to a quadratic form $F$ on $H$ if the following conditions hold:
\begin{enumerate}
\item[(i)] if $(x_m)_{m\in\N}\in H^\N$ and $x\in H$ are such that $x_m\rightharpoonup x$, then
\begin{equation*}
F(x)\le \liminf_{m\to\infty}F_m(x_m);
\end{equation*}
\item[(ii)] for all $x\in H$, there exists $(x_m)_{m\in\N}\in H^\N$ such that $x_m\to x$ and
\begin{equation*}
F(x) = \lim_{m\to\infty}F_m(x_m).
\end{equation*}
\end{enumerate}
\end{definition}

To prove such convergence, if $\Omega$ is an $H^1$-extension domain, we denote by $\mathrm{E}_\Omega:H^1(\Omega)\to H^1(U)$ a linear bounded $H^1$-extension operator. If $u\in L^2(]0,T[,H^1(\Omega))$, the extension $\mathrm{E}_\Omega u$ is understood as follows:
\begin{equation}\label{Eq:SpaceExtension}
\mathrm{E}_\Omega u:t\in [0,T]\longmapsto \big(x\in U\longmapsto [\mathrm{E}_\Omega u(t)](x)\big).    
\end{equation}
In that sense, $\mathrm{E}_\Omega$ is also regarded as $L^2(]0,T[,H^1(\Omega))\to L^2(]0,T[,H^1(U))$.

\begin{theorem}[Mosco convergence of the energy forms]\label{Th:Mosco}
Let $\varepsilon > 0$. Let $(\Omega_m)_{m\in\N}$ be sequence of ($\varepsilon$,$\infty$)-domains included in $U$ and $(\mu_{m})_{m\in\N}$ be a sequence of finite Borel measures with $\mathrm{supp}\,\mu_m=\bord_m$, all satisfying~\eqref{Eqmu} for the same $d\in[n-1,n[$ and the same constant $c_d>0$. Assume there exist a domain $\Omega$ and a measure $\mu$ with $\mathrm{supp}\,\mu=\bord$ such that:
\begin{enumerate}
\item[(i)] $(\Omega_m)_{m\in\N}$ converges to $\Omega$ in the sense of characteristic functions and Hausdorff (denoted by $\Omega_m\xrightarrow[]{\mathds{1},H}\Omega$);
\item[(ii)] $(\mu_m)_{m\in\N}$ converges to $\mu$ weakly (denoted by $\mu_m\rightharpoonup\mu$), \textit{i.e.,}
\begin{equation*}
\forall \varphi\in C\big(\overline{U},\R\big),\quad \int_{\bord_m}{\varphi\,\dmu_m}\xrightarrow[m\to\infty]{}\int_{\bord}{\varphi\,\dmu}.
\end{equation*}
\end{enumerate}
Then,
\begin{equation}
\lim_{m \to \infty}J_T(\Omega_{m},\mu_{m})=J_T(\Omega,\mu) \quad\mbox{in the sense of Mosco.}
\end{equation}
\end{theorem}

\begin{proof}
We adapt the proof of~\cite[Theorem 8]{hinz_non-lipschitz_2021}, where the energy was time-independant and the boundary term was a trace value instead of a jump. Let $(u_m)\in L^2(]0,T[,L^2(U))^\N$ and $u\in L^2(]0,T[,L^2(U))$ such that $u_m\rightharpoonup u$. We aim at proving condition (i) from Definition~\ref{Def:MoscoCV}, that is:
\begin{equation}\label{Eq:Mosco1}
J_T(\Omega,\mu)(u)\le \underset{m\to\infty}{\lim\inf} J_T(\Omega_m,\mu_m)(u_m).
\end{equation}
Up to passing to a subsequence, we may assume:
\begin{equation*}
\liminf_{m\in\N} J_T(\Omega_m,\mu_m)(u_m)=\lim_{m\to\infty} J_T(\Omega_m,\mu_m)(u_m),
\end{equation*}
and
\begin{equation*}
\forall m\in\N,\quad u_m|_{\Omega_m}\in L^2(]0,T[,H^1(\Omega_m)),
\end{equation*}
so that the limit is finite. Then, from the definition of $J_T$, it holds:
\begin{equation}\label{Eq:Mosco1-bounded}
\sup_{m\in\N} \|u_m|_{\Omega_m}\|_{L^2(]0,T[,H^1(\Omega_m))}<+\infty. 
\end{equation}
By~\cite[Proposition 4, Remark 8]{hinz_non-lipschitz_2021}, there exists a sequence of linear bounded extension operators $(\mathrm{E}_{\Omega_m}:H^1(\Omega_m)\to H^1(U))_{m\in\N}$ and a constant $c_{\mathrm{E}}>0$ such that:
\begin{equation*}
\forall m\in\N,\quad \|\mathrm{E}_{\Omega_m}(u_m|_{\Omega_m})\|_{L^2(]0,T[,H^1(U))}\le c_{\mathrm{E}}\|u_m|_{\Omega_m}\|_{L^2(]0,T[,H^1(\Omega_m))},
\end{equation*}
hence, by~\eqref{Eq:Mosco1-bounded}, $(\mathrm{E}_{\Omega_m}(u_m|_{\Omega_m}))_{m\in\N}$ (in the sense of~\eqref{Eq:SpaceExtension}) is bounded in the Hilbert space $L^2(]0,T[,H^1(U))$.
Up to passing to a subsequence, we may assume there exists $u_\infty\in L^2(]0,T[,H^1(U))$ such that $\mathrm{E}_{\Omega_m}(u_m|_{\Omega_m})\rightharpoonup u_\infty$ in $L^2(]0,T[,H^1(\Omega_m))$. From the convergence of the domains in the sense of characteristic functions, we may also assume $\mathds{1}_{\Omega_m}\to \mathds{1}_\Omega$ a.e., hence:
\begin{align*}
\mathds{1}_{\Omega_m}\mathrm{E}_{\Omega_m}(u_m|_{\Omega_m}) &\rightharpoonup \mathds{1}_\Omega u_\infty & &\hspace{-4em}\mbox{in }L^2(]0,T[,L^2(U)),\\
\mathds{1}_{\Omega_m}\nabla\mathrm{E}_{\Omega_m}(u_m|_{\Omega_m}) &\rightharpoonup \mathds{1}_\Omega \nabla u_\infty & &\hspace{-4em}\mbox{in }L^2(]0,T[,L^2(U)^n).
\end{align*}
Since $\mathds{1}_{\Omega_m}\mathrm{E}_{\Omega_m}(u_m|_{\Omega_m})=\mathds{1}_{\Omega_m}u_m\rightharpoonup \mathds{1}_\Omega u$ in $L^2(]0,T[,L^2(U))$ (see also~\cite[Paragraph 7.1.2, Theorem 3]{evans_partial_2010} for the initial condition), we can deduce:
\begin{equation}\label{Eq:IdentificationInt}
u_\infty|_\Omega=u|_\Omega.
\end{equation}
In the same way, up to passing to a subsequence, we can assume that for all $m\in\N$, $u_m|_{U\backslash\overline{\Omega}_m}\in L^2(]0,T[,H^1(U\backslash\overline{\Omega}_m))$, where $\overline{\Omega}_m$ denotes the closure of $\Omega_m$, and obtain:
\begin{equation}\label{Eq:IdentificationExt}
\mathds{1}_{U\backslash\overline{\Omega}_m}\,u_m\rightharpoonup \mathds{1}_{U\backslash\overline{\Omega}}\,u\quad\mbox{in }L^2(]0,T[,H^1(U)).
\end{equation}

Using a similar decomposition as in the proof of~\cite[Theorem 5]{hinz_non-lipschitz_2021} (\textit{i.e.}, approximating the $H^1$-extensions of $u$ with smooth functions) yields:
\begin{align}\label{Eq:BoundaryIntCV}
\int_0^T\!\!\int_{\bord_m}{\Lambda|\!\Tr^+_{\bord_m}\!u_m|^2\,\dmu_k\,\dt} &\longrightarrow \int_0^T\!\!\int_{\bord}{\Lambda|\!\Tr^+_{\bord}\!u|^2\,\dmu\,\dt},\nonumber\\
\int_0^T\!\!\int_{\bord_m}{\Lambda|\!\Tr^-_{\bord_m}\!u_m|^2\,\dmu_k\,\dt} &\longrightarrow \int_0^T\!\!\int_{\bord}{\Lambda|\!\Tr^-_{\bord}\!u|^2\,\dmu\,\dt},\\
\int_0^T\!\!\int_{\bord_m}{\Lambda(\Tr^+_{\bord_m}\!u_m)(\Tr^-_{\bord_m}\!u_m)\,\dmu_k\,\dt} &\longrightarrow \int_0^T\!\!\int_{\bord}{\Lambda(\Tr^+_{\bord}\!u)(\Tr^-_{\bord}\!u)\,\dmu\,\dt},\nonumber
\end{align}
The limits from Eq.~\eqref{Eq:BoundaryIntCV} along with the lower semi-continuity of the weak limits from Eqs.~\eqref{Eq:IdentificationInt} and~\eqref{Eq:IdentificationExt} yield~\eqref{Eq:Mosco1}.

To complete this proof, given $u\in L^2(]0,T[,L^2(U))$, we seek a sequence $(u_m)\in L^2(]0,T[,L^2(U))^\N$ such that $u_m\to u$ and $J_T(\Omega_m,\mu_m)(u_m)\to J_T(\Omega,\mu)(u)$. We may assume $u|_{U\backslash\bord}\in L^2(]0,T[,H^1(U\backslash\bord))$ so that the limit is finite. We define
\begin{equation*}
u_m:=(\mathrm{E}_\Omega u)|_{\Omega_m}+(\mathrm{E}_{U\backslash\overline{\Omega}}u)|_{U\backslash\overline{\Omega}_m},
\end{equation*}
where $\mathrm{E}_\Omega:H^1(\Omega)\to H^1(U)$ and $\mathrm{E}_{U\backslash\overline{\Omega}}:H^1(U\backslash\overline{\Omega})\to H^1(U)$ are linear continuous extension operators. Then, by dominated convergence,
\begin{equation*}
\int_0^T\!\!\int_{\Omega_m}{\big(|u_m|^2+|\nabla u_m|^2\big)\,\dx\,\dt}\longrightarrow \int_0^T\!\!\int_\Omega{\big(|u|^2+|\nabla u|^2\big)\,\dx\,\dt}.
\end{equation*}
That limit, along with the limits of boundary integrals given by Eq.~\eqref{Eq:BoundaryIntCV} yield $J_T(\Omega_m,\mu_m)(u_m)\to J_T(\Omega,\mu)(u)$, which is condition (ii) from Definition~\ref{Def:MoscoCV}.
\end{proof}

Given the definition of convergence in the sense of Mosco and Theorem~\ref{Th:Mosco}, we can deduce the $\Gamma$-convergence of the energy forms. In particular, by~\cite[Corollary 7.20]{dal_maso_introduction_1993}, the convergence of the energy forms yields the following result on convergence of minimizers. Note that those minimizers depend notably on the diffusion time $T>0$.

\begin{corollary}[Convergence of minimizers]\label{Cor:MinCV}
Let $\varepsilon > 0$. Let $(\Omega_m)_{m\in\N}$ be sequence of ($\varepsilon$,$\infty$)-domains included in $U$ and $(\mu_{m})_{m\in\N}$ be a sequence of finite Borel measures with $\mathrm{supp}\,\mu_m=\bord_m$, all satisfying~\eqref{Eqmu} for the same $d\in[n-1,n[$ and the same constant $c_d>0$. Assume there exists a domain $\Omega$ and a measure $\mu$ with $\mathrm{supp}\,\mu=\bord$ such that:
\begin{enumerate}
\item[(i)] $(\Omega_m)_{m\in\N}$ converges to $\Omega$ in the sense of characteristic functions and Hausdorff;
\item[(ii)] $(\mu_m)_{m\in\N}$ converges to $\mu$ weakly.
\end{enumerate}
If $(u_m)_{m\in\N}\in L^2(]0,T[,L^2(U))^\N$ is such that for all $m\in\N$, $u_m$ is a minimizer of $J_T(\Omega_m,\mu_m)$, then any accumulation point $u$ of $(u_m)_{m\in\N}$ is a minimizer of $J_T(\Omega,\mu)$ and it holds:
\begin{equation*}
J_T(\Omega,\mu)(u)=\limsup_{m\to\infty} J_T(\Omega_m,\mu_m)(u_m).
\end{equation*}
If $(u_m)_{m\in\N}$ converges to $u$ in $L^2(]0,T[,L^2(U))$, then $u$ is a minimizer of $J_T(\Omega,\mu)$ and it holds:
\begin{equation*}
J_T(\Omega,\mu)(u)=\lim_{m\to\infty} J_T(\Omega_m,\mu_m)(u_m).
\end{equation*}
\end{corollary}

Consequently, if $(\Omega_m,\mu_m)_{m\in\N}$ is a sequence of $(\varepsilon,\infty)$-domains for some $\varepsilon>0$ converging to $(\Omega,\mu)$ in the sense of Theorem~\ref{Th:Mosco}, then energy minimizers for $(\Omega_m,\mu_m)_{m\in\N}$ allow to determine energy minimizers for the limit domain. In the next section, we prove the existence of an optimal shape in a class of $(\varepsilon,\infty)$-domains. In particular, that shape can be fractal: Corollary~\ref{Cor:MinCV} allows to approximate an energy minimizer $u$ (and the associated energy) on the optimal shape with minimizers on regular shapes.

\subsection{Shape optimization}\label{Sec:Opt}

We wish to prove the existence of a domain $\Omega$ along with a boundary measure which maximize the heat transfer from $\Omega^+$ to $\Omega^-$ in a class of domains of fixed volume. To that end, we consider the following energy functional:
\begin{equation}\label{Eq:EnergyFunctional}
\J_T(\Omega,\mu):=J_T(\Omega,\mu)(u_{(\Omega,\mu)}),
\end{equation}
where $(\Omega,\mu)$ is a two-sided admissible domain of $\R^n$, $J_T$ is defined by~\eqref{functionaltime} and $u_{(\Omega,\mu)}$ is the unique weak solution to problem~\eqref{Eq:Problem} for $(\Omega,\mu)$ (in the sense of Theorem~\ref{Th:WP}). To prove the existence of an optimal shape $(\Omega_*,\mu_*)$:
\begin{equation*}
\J_T(\Omega_*,\mu_*)=\min_{(\Omega,\mu)\in\U}\J_T(\Omega,\mu),
\end{equation*}
we consider several classes of admissible domains $\U$: a class of Lipschitz domains on the one hand (Paragraph~\ref{Subsubsec:OptLip}), and a class of uniform domains on the other hand (Paragraph~\ref{Subsubsec:OptUnif}).

\subsubsection{In a class of Lipschitz domains}\label{Subsubsec:OptLip}

We start with seeking an optimal shape in a class of Lipschitz subdomains of $U$ (which is still a large bounded Lipschitz domain). More specifically, we wish to prove the existence of an optimal domain $\Omega$ which satisfies the $\varepsilon$-cone property~\cite[Definition 2.4.1]{henrot_shape_2018} for a given $\varepsilon>0$, \textit{i.e.} for all $x\in\bord$, there exists $\xi_x\in\R^n$ unitary such that for all $y\in\overline{\Omega}\cap B_\varepsilon(x)$, it holds:
\begin{equation*}
\big\{z\in\R^n\;|\;\langle z-y,\xi_x\rangle >\cos(\varepsilon)|z-y| \; \mbox{ and }\; 0<|z-y|<\varepsilon \big\} \subset \Omega.
\end{equation*}
A bounded domain $\Omega$ of $\R^n$ is Lipschitz if and only if it satisfies the $\varepsilon$-cone property for some $\varepsilon>0$~\cite[Theorem 2.4.7]{henrot_shape_2018}.
For $\varepsilon>0$, let $\Theta(U,\varepsilon)$ be the set of subdomains of $U$ with connected boundary satisfying the $\varepsilon$-cone property.

We also require the domain $\Omega$ to be of fixed volume (with respect to $\lambda^{(n)}$), its boundary to be confined in $\overline{G}$ where $G\subsetneq U$ is a domain of $\R^n$ and such that, for some $\hat{c}>0$, it holds:
\begin{equation}
\forall r>0,\;\forall x\in\bord,\quad \Hcal^{(n-1)}(\bord\cap B_r(x))\le \hat{c}\,r^{n-1},
\end{equation}
where $\Hcal^{(n-1)}$ is the $(n-1)$-dimensional Hausdorff measure (here, with support $\bord$). Note that, by boundedness of $U$, this implies there exists $M$ depending on $\hat{c}$ such that $\Hcal^{(n-1)}(\bord)\le M$.

From those conditions, we may define the class of admissible Lipschitz domains as follows (see Figure~\ref{Fig:Admissible}).

\begin{definition}[Admissible Lipschitz domains]
Let $\varepsilon>0$, $V>0$ and $\hat{c}>0$. Let $G\subsetneq U$ be a bounded domain of $\R^n$. The class of admissible Lipschitz domains is defined by:
\begin{multline}\label{Eq:UadLip}
\U_{U\!,\varepsilon}(G, V, \hat{c}):=\big\{\Omega\in\Theta(U,\varepsilon)\;\big|\;\lambda^{(n)}(\Omega)=V,\quad\bord\subset\overline{G},\\
\mbox{and}\quad \forall r>0,\;\forall x\in\bord,\:\: \Hcal^{(n-1)}(\bord\cap B_r(x))\le \hat{c}\,r^{n-1}\big\}.
\end{multline}
\end{definition}

\begin{figure}[ht]
\centering
\includegraphics[scale=0.5]{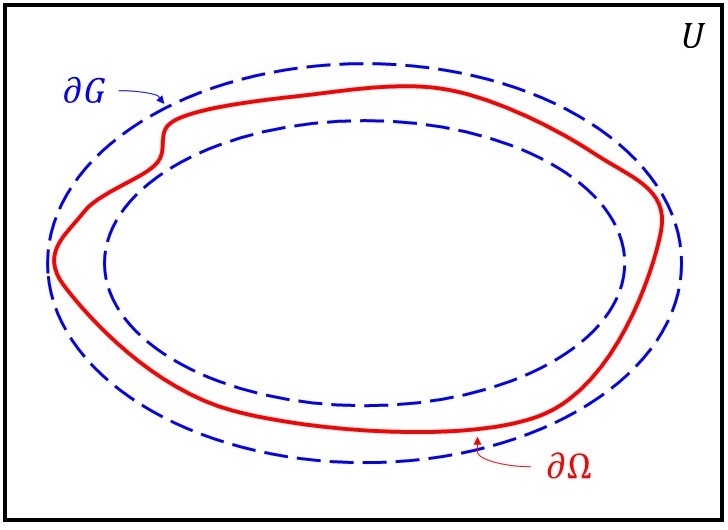}
\caption{An example of an admissible Lipschitz domain, lying inside the red boundary $\bord$. $U$ is the interior of the black rectangle. $G$ lies between the blue dashed lines $\partial G$.}
\label{Fig:Admissible}
\end{figure}

A key ingredient in optimization problems is the compactness of the class of admissible domains with respect to some domain topology. In our case, the class of admissible Lipschitz domains is compact with respect to the convergence in the sense of characteristic functions, Hausdorff and compact sets~\cite[Definitions 2.2.3, 2.2.8 and 2.2.21]{henrot_shape_2018}. However, it is not compact with respect to the weak convergence of boundary measures, in the sense that the weak limit of a sequence of Hausdorff measures is not necessarily a Hausdorff measure itself.

\begin{theorem}\label{Th:Compactness}
Let $\varepsilon>0$, $V>0$ and $\hat{c}>0$. Let $G\subsetneq U$ be a bounded domain of $\R^n$. Then, the class of admissible Lipschitz domains $\U_{U\!,\varepsilon}(G, V, \hat{c})$ is compact for the convergence in the sense of characteristic functions, Hausdorff and compact sets.
\end{theorem}

\begin{proof}
We follow the proof of~\cite[Lemma 3.1]{magoules_optimal_2021}, in dimension $n\ge 2$ and considering the Hausdorff measure $\Hcal^{(n-1)}$ instead of the Lebesgue measure $\lambda^{(n-1)}$ (which coincide for $n=2$). Such a modification can be seen as rescaling the constant $\hat{c}$.
\end{proof}

Since the class of admissible Lipschitz domains is compact with respect to domain convergence, it is possible to find optimal shapes in it for problems with homogeneous Dirichlet or Neumann boundary conditions. However, in our case, the energy form $\J_T$ also takes the boundary measure into account and we prove in the following theorem that although an optimal geometry exists in that class, its boundary is not necessarily endowed with the Hausdorff measure. Considering a boundary measure which is not the Hausdorff measure for a Lipschitz domain can be seen as measuring the perimeter of the volume in an unusual manner, in particular one which is not consistent with the physical modelling of heat diffusion problems. Note that the optimal shapes discussed in the following theorem depend on the parameters which were set at the beginning of this section, notably the transfer time $T$.

\begin{theorem}[Lipschitz optimal shape]\label{Th:ShapeOptLip}
Let $\varepsilon>0$, $V>0$ and $\hat{c}>0$. Let $G\subsetneq U$ be a bounded domain of $\R^n$. Then, there exists $\Omega_* \in\U_{U\!,\varepsilon}(G, V, \hat{c})$ and a finite $(n-1)$-dimensional positive measure $\mu_*$ on $\bord_*$ equivalent to the Hausdorff measure $\mathcal{H}^{(n-1)}$ such that, for all Borel set $\Gamma\subset\bord_*$:
\begin{equation}\label{Eq:Mu*Lip}
\mu_*(\Gamma) \ge \mathcal{H}^{(n-1)}(\Gamma),
\end{equation}
and
\begin{multline*}
J_T\big(\Omega_*, \mathcal{H}^{(n-1)}\big)(u_{(\Omega_*\!,\,\mu_*)})\le \inf_{\Omega \in\, \U_{U\!,\varepsilon}(G, V, \hat{c})} \J_T\big(\Omega, \mathcal{H}^{(n-1)}\big)
=\J_T\left(\Omega_*, \mu_*\right),
\end{multline*}
where $\J_T$ is defined by~\eqref{Eq:EnergyFunctional} and $u_{(\Omega_*\!,\,\mu_*)}$ denotes the solution to problem~\eqref{Eq:Problem} on $(\Omega_*,\mu_*)$ in the sense of Theorem~\ref{Th:WP}.
\end{theorem}

\begin{proof}
Let $(\Omega_m)_{m\in\N}\in \U_{U\!,\varepsilon}(G, V, \hat{c})^\N$ be such that:
\begin{equation*}
\J_T(\Omega_m,\Hcal^{(n-1)})\xrightarrow[m\to\infty]{}\inf_{\Omega\in\, \U_{U\!,\varepsilon}(G, V, \hat{c})}\J_T(\Omega,\Hcal^{(n-1)})\ge0.
\end{equation*}
Then, by compactness of the class of admissible Lipschitz domains (Theorem~\ref{Th:Compactness}), we may assume (up to passing to a subsequence) there exists $\Omega_*\in\U_{U\!,\varepsilon}(G, V, \hat{c})$ such that
\begin{equation*}
\Omega_m\xrightarrow[m\to\infty]{\mathds{1},H,K}\Omega_*,
\end{equation*}
where $K$ stands for convergence in the sense of compact sets. Minor changes in the proof of~\cite[Lemma 3.1, Point 3]{magoules_optimal_2021} (\textit{i.e.}, considering subsets of the boundary instead of the whole boundary) yields the existence of $\mu_*$ supported by $\bord_*$ satisfying~\eqref{Eq:Mu*Lip}.
For $m\in\N$, let $u_m$ be the solution to problem~\eqref{Eq:Problem} on $(\Omega_m,\Hcal^{(n-1)})$ in the sense of Theorem~\ref{Th:WP}. Since the problem is well-posed (see Eq.~\eqref{Eq:EnergyEstimate}) and by uniform boundedness of the extension operators $\mathrm{E}_{\Omega_m}:H^1(\Omega_m)\to H^1(U)$, $(\mathrm{E}_{\Omega_m}u_m)_{m\in\N}$ forms a bounded sequence of $L^2(]0,T[,H^1(U))$ and $(\mathrm{E}_{\Omega_m}(\partial_t u_m))_{m\in\N}$ forms a bounded sequence of $L^2(]0,T[,H^1(U)')$. Hence, there exists $u_\infty$ such that (up to passing to a subsequence):
\begin{align*}
u_m&\xrightharpoonup[m\to\infty]{}u_\infty & &\hspace{-5em}\mbox{in }L^2(]0,T[,H^1(U)),\\
\partial_t u_m&\xrightharpoonup[m\to\infty]{}\partial_t u_\infty & &\hspace{-5em}\mbox{in }L^2(]0,T[,H^1(U)').
\end{align*}
Proceeding as in the proof of Theorem~\ref{Th:Mosco} leading to Eq.~\eqref{Eq:IdentificationInt}, we prove $u_\infty|_\Omega=u_*|_\Omega$, where $u_*$ is the weak solution to problem~\eqref{Eq:Problem} on $(\Omega_*,\mu_*)$. Up to passing to a subsequence, we may assume the convergence in the sense of characteristic functions in an a.e. pointwise convergence. Hence, by dominated convergence (for the integrals on the bulk) and using the boundary integral convergences from Eq.~\eqref{Eq:BoundaryIntCV}, we can deduce:
\begin{equation*}
\J_T(\Omega_m,\Hcal^{(n-1)})=J_T(\Omega_m,\Hcal^{(n-1)})(u_m)\xrightarrow[m\to\infty]{}J_T(\Omega_*,\mu_*)(u_*)=\J_T(\Omega_*,\mu_*).
\end{equation*}
Since $\mu_*\ge\Hcal^{(n-1)}$, it follows that:
\begin{equation*}
J_T(\Omega_*, \mathcal{H}^{(n-1)})(u_{(\Omega_*\!,\, \mu_*)}) \le \J_T\left(\Omega_*, \mu_*\right).
\end{equation*}
\end{proof}

\subsubsection{In a class of uniform domains}\label{Subsubsec:OptUnif}

As it was proved in Theorem~\ref{Th:ShapeOptLip}, the measure on the boundary of an optimal Lipschitz geometry is not necessarily the Hausdorff measure $\Hcal^{(n-1)}$, which means, in some sense, that the energy infimum can be reached outside of the class of Lipschitz admissible boundaries (implicitly endowed with the Hausdorff measure as explained before). For that matter, we extend the class of admissible domains: as of now, we wish to find an optimal domain $\Omega$ which is an $(\varepsilon,\infty)$-domain for some $\varepsilon>0$, \textit{i.e.}, for all $x,y\in\Omega$, there exists a rectifiable arc $\gamma$ drawn on $\Omega$ and joining $x$ and $y$, of length $\ell(\gamma)$ and such that:
\begin{enumerate}
\item[(i)] $\ell(\gamma)\le\frac{|x-y|}{\varepsilon}$;
\item[(ii)] for all $z\in\gamma$,\; $d(z,\bord)\ge\varepsilon\frac{|x-z||y-z|}{|x-y|}$.
\end{enumerate}
For $\varepsilon>0$, let $\Ocal(U,\varepsilon)$ be the set of $(\varepsilon,\infty)$-domains included in $U$ and with a connected boundary. We also require the boundary to be the support of a measure $\mu$ satisfying the $d$-upper regularity condition~\eqref{Eqmu} for some $d\in[n-1,n[$ and $c_d>0$, as well as the $s$-lower regularity condition for some $s\in[n-1,n[$ and $c_s>0$:
\begin{equation}\label{Eq:muLower}
\forall x\in\bord,\quad \forall r\in]0,1],\quad c_s\,r^s\le\mu(\overline{B_r(x)}),
\end{equation}
which means the Hausdorff dimension of the boundary cannot exceed $s$.
Naturally, it must hold $d\le s$ for such a measure to exist. From those conditions, we may define the class of admissible uniform domains as follows (see Figure~\ref{Fig:UnifAd}).

\begin{definition}[Admissible uniform domains]
Let $\varepsilon>0$, $V>0$, $d,s\in[n-1,n[$ with $d\le s$ and $c_d,c_s>0$. Let $G\subsetneq U$ be a bounded domain of $\R^n$. The class of admissible uniform domains is defined by:
\begin{multline}
\U_{U\!,\varepsilon}^*(G, V, d, s, c_d, c_s):=\big\{(\Omega,\mu)\;\big|\;\Omega\in\Ocal(U,\varepsilon),\quad\lambda^{(n)}(\Omega)=V,\quad\bord\subset\overline{G},\\
\mathrm{supp}\,\mu=\bord\quad\mbox{and}\quad\mu\mbox{ satisfies }\eqref{Eqmu}\mbox{ and }\eqref{Eq:muLower}\big\}.
\end{multline}
\end{definition}

\begin{figure}[ht]
    \centering
    \includegraphics[scale=0.5]{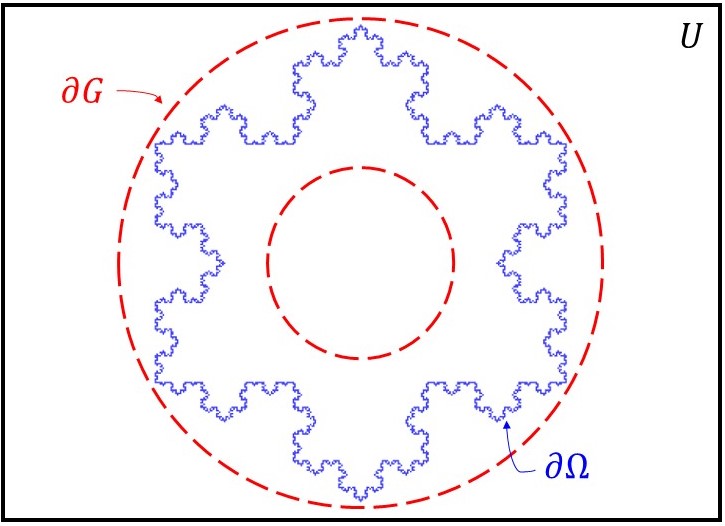}
    \caption{An example of a admissible uniform domain which is not an admissible Lipschitz domain. The domain lies inside the blue Von Koch curve (represented by a fifth generation pre-fractal {\small[source: Wikipedia]}) and its boundary is endowed with the $\frac{\ln 4}{\ln 3}$-dimensional Hausdorff measure. $U$ is the interior of the black rectangle. $G$ lies between the red dashed lines $\partial G$.}
    \label{Fig:UnifAd}
\end{figure}

It follows from~\cite[Theorem 3]{hinz_non-lipschitz_2021} that the class of admissible uniform domains is compact with respect to the convergence in the sense of characteristic functions, Hausdorff and compact sets, but also with respect to the weak convergence of boundary measures.

\begin{theorem}\label{Th:CompactUniform}
Let $\varepsilon>0$, $V>0$, $d,s\in[n-1,n[$ with $d\le s$ and $c_d,c_s>0$. Let $G\subsetneq U$ be a bounded domain of $\R^n$. Then,
\begin{enumerate}
\item[(i)] the class $\U_{U\!,\varepsilon}^*(G, V, d, s, c_d, c_s)$ of admissible uniform domains is compact in the sense of characteristic functions, Hausdorff and compact sets, as well as in the sense of weak convergence of boundary measures.

\item[(ii)] if $(\Omega_m,\mu_m)_{m \in \N}\in \U_{U\!,\varepsilon}^*(G, V, d, s, c_d, c_s)^\N$ is such that $(\mu_m)_{m\in\N}$ converges weakly, then $(\Omega_{m})_{m \in \N}$ converges in the sense of characteristic functions, Hausdorff and compact sets. 
\end{enumerate}
\end{theorem}

In light of Theorem~\ref{Th:CompactUniform}, the proof of Theorem~\ref{Th:ShapeOptLip} can be adapted on the class of admissible uniform domains to prove the existence of an optimal shape. Note that, once again, the optimal shapes discussed in the following theorem depend on the parameters which were set at the beginning of this section, notably the transfer time $T$.

\begin{theorem}[Uniform optimal shape]\label{Th:ShapeOptUnif}
Let $\varepsilon>0$, $V>0$, $d,s\in[n-1,n[$ with $d\le s$ and $c_d,c_s>0$. Let $G\subsetneq U$ be a bounded domain of $\R^n$. Then, there exists $(\Omega_*,\mu_*) \in\U_{U\!,\varepsilon}^*(G, V, d, s, c_d, c_s)$ such that
\begin{equation*}
\J_T(\Omega_*, \mu_*)=\min_{(\Omega,\mu) \in\, \U_{U\!,\varepsilon}^*(G, V, d, s, c_d, c_s)} \J_T(\Omega, \mu),
\end{equation*}
where $\J_T$ is defined by~\eqref{Eq:EnergyFunctional}.
\end{theorem}

\section*{Acknowledgments}
The authors are grateful to Gabriel Singer for his preliminary work on the topic, conducted during an internship supervised by Anna Rozanova-Pierrat. The authors also thank Alexander Teplyaev and Michael Hinz for enlightening discussions on the topic. The motivation for this work originated from research conducted by Anna Rozanova-Pierrat with Bernard Sapoval and Denis Grebenkov in \'Ecole Polytechnique, France.

\bibliographystyle{alpha}
\def\refname{References}
\bibliography{BibGC.bib}

\end{document}